\documentclass{amsart}
\usepackage{amsfonts}
\usepackage{amssymb}
\usepackage{times}

\newtheorem{lemma}{Lemma}
\newtheorem{de}{Definition}
\newtheorem{teo}{Theorem}

\newfont{\hueca}{msbm10}

\begin{document}
\title[Leibniz algebras associated with representations of filiform Lie algebras]
{Leibniz algebras associated with representations of filiform Lie algebras}

\author{Sh.A. Ayupov, L.M. Camacho, A.Kh. Khudoyberdiyev, B.A. Omirov}
\address{[Sh.A. Ayupov --- A.Kh. Khudoyberdiyev --- B.A. Omirov] Institute of Mathematics. National
University of Uzbekistan, Dormon yoli str. 29, 100125, Tashkent
(Uzbekistan)}
\email{sh\_ayupov@mail.ru --- khabror@mail.ru --- omirovb@mail.ru}
\address{[L.M. Camacho ] Dpto. Matem\'{a}tica Aplicada I.
Universidad de Sevilla. Avda. Reina Mercedes, s/n. 41012 Sevilla.
(Spain)} \email{lcamacho@us.es}

%\address{}%
%\email{}%
%
\thanks{The third named author was partially supported by IMU/CDC-program, and he also supported
by Instituto de Mat\'{e}maticas de la Universidad de Sevilla. The
authors were supported by Ministerio de Econom\'{\i}a y
Competitividad
(Spain), grant MTM2013-43687-P (European FEDER support included).}%
%\subjclass{17A32, 17B30, 17B10}%
%\keywords{Filiform algebra, Heisenberg algebra, Fock representation, minimal faithful representation}%

%\date{}%
%\dedicatory{The firth author...}%
%\commby{}%

\maketitle
\begin{abstract}
In this paper we investigate Leibniz algebras whose quotient Lie algebra is a naturally graded filiform Lie algebra $n_{n,1}.$ We introduce a Fock module for the algebra $n_{n,1}$ and provide classification of Leibniz algebras $L$ whose corresponding Lie algebra $L/I$ is the algebra $n_{n,1}$ with condition that the ideal $I$ is a Fock $n_{n,1}$-module, where $I$ is the ideal generated by squares of elements from $L$.

We also consider Leibniz algebras with  corresponding Lie algebra $n_{n,1}$ and such that the action $I \times n_{n,1} \to I$ gives rise to a minimal faithful representation of $n_{n,1}$. The classification up to isomorphism of such Leibniz algebras is given for the case of $n=4.$
\end{abstract}

\medskip \textbf{AMS Subject Classifications (2010):
17A32, 17B30, 17B10.}

\textbf{Key words:}  Filiform algebra, Heisenberg algebra, Fock representation, minimal faithful representation.

\section{Introduction}

Leibniz algebras are  generalizations of Lie algebras
and they have
been firstly introduced by Loday in \cite{Loday} as a
non-antisymmetric version of Lie algebras.

However this kind of algebras was previously introduced and studied under the notion of $D$-algebras by D. Bloh \cite{B1}. Since the 1993 when Loday's work was published, many researchers have been
attracted to Leibniz algebras, with remarkable activity  during the last decade. Namely, the investigations have been mainly focused on low dimensional, nilpotent, solvable and other special classes of algebras (see \cite{Dzhu, Ayu01, Cabezas, ca4, CanKhud, le12, GomVid, ma2}).

Recall that the variety of Leibniz algebras is defined by the fundamental identity
$$[x,[y,z]]=[[x,y],z] - [[x,z],y].$$

In fact,  each non-Lie Leibniz algebra $L$ contains a non-trivial ideal (further denoted by $I$), which is the subspace spanned by squares of elements of the algebra $L$ .
 Moreover, it is readily to see that this ideal belongs to right annihilator of $L$, that is $[L,I]=0$. Note also that the ideal $I$ is the minimal ideal with the property that the quotient algebra $L/I$ is a Lie algebra.

One of the approaches to investigation of Leibniz algebras is a description of such algebras whose quotient algebra with respect to the ideal $I$ is a given Lie algebra.
In particular in  \cite{huel3} the description has been obtained for finite-dimensional complex Leibniz algebras whose quotient algebra is isomorphic to the simple Lie algebra $sl_2$.
In \cite{Bar}  D. Barnes showed that any finite-dimensional complex Leibniz algebra can be decomposed into a semidirect sum of the  solvable radical and a semi-simple Lie algebra (the analogue of Levi's theorem). Hence, we conclude that if the quotient algebra is isomorphic to a semi-simple Lie algebra, then knowing a module over this semi-simple Lie algebra, one can easily obtain the description of Leibniz algebras with this properties.

Therefore, it is important to study the case when the quotient Lie algebra is solvable, or moreover is nilpotent. Since the Heisenberg and filiform Lie algebras are well-known, it is natural to consider a Leibniz algebra whose  quotient Lie algebra is the Heisenberg algebra $H_n$ or the filiform Lie algebra $n_{n,1}$.  On the other hand, we recall that Heisenberg and filiform Lie algebras play an important role in mathematical physics and geometry, in particular in Quantum Mechanics (see for instance \cite{deJeu, Gelou, chinos, Konstantina}). Indeed, the Heisenberg Uncertainty Principle implies the non-compatibility of position and momentum observables acting on fermions. In \cite{Cal} some Leibniz algebras with the quotient algebra being Heisenberg algebra are described. In particular, a classification theorem was obtained for Leibniz algebras whose corresponding Lie algebra is $H_n$ and that the ${H_n}$-module $I$ is isomorphic to its Fock module.

In order to achieve the goal of our study we organize the paper as follows. The first two sections are devoted to introduction and preliminaries. In section 3 we introduce the Fock module for the filiform Lie algebra $n_{n,1}$ and give the classification of Leibniz algebras with corresponding Lie algebra $n_{n,1}$ under the condition that the   $n_{n,1}$-module $I$ is Fock module. In this section we also consider a generalization of this class of
algebras by considering the direct sum of filiform Lie algebras as the corresponding Lie algebras, and provide a classification theorem. Finally, in Section 4 we deal with the category of Leibniz algebras with $n_{n,1}$ as corresponding Lie algebra and such that the action $I \times n_{n,1} \to I$ gives rise to a minimal faithful
representation of $n_{n,1}$. A complete description of this category of algebras is given when dimension is equal to 4.

\section{Preliminaries}

In this section we give necessary definitions and preliminary results.

\begin{de} An algebra $(L,[-,-])$ over a field  $\mathbb{F}$   is called a Leibniz algebra if for any $x,y,z\in L$, the so-called Leibniz identity
\[ \big[[x,y],z\big]=\big[[x,z],y\big]+\big[x,[y,z]\big] \] holds.
\end{de}

For a Leibniz algebra $L$ consider the following  lower central series:
\[L^1=L,\qquad L^{k+1}=[L^k,L^1] \qquad k \geq 1.\]

Since the notions of right nilpotency and nilpotency coincide, we can define nilpotency as follows:
\begin{de} A Leibniz algebra $L$ is called nilpotent if there exists  $s\in\mathbb N $ such that $L^s=0$.
\end{de}

\begin{de} A Leibniz algebra $L$ is said to be filiform if $\dim L^i=n-i$, where $n=\dim L$ and $2\leq i \leq n$.
\end{de}

Now let us define a natural gradation for a filiform Leibniz algebra.

\begin{de} Given a filiform Leibniz algebra $L$, put $L_i=L^i/L^{i+1}, \ 1 \leq i\leq n-1$, and $Gr(L) = L_1 \oplus
L_2\oplus\dots \oplus L_{n-1}$. Then $[L_i,L_j]\subseteq L_{i+j}$ and we obtain the graded algebra $Gr(L)$. If $Gr(L)$ and $L$ are isomorphic, then we say that the algebra $L$ is naturally graded.
\end{de}

From \cite{Ver} it is well known that there are two types of naturally graded filiform Lie algebras. In fact, the second type will appear only in the case when the dimension of the algebra is even.

\begin{teo}[\cite{Ver}] \label{thm2.8} Any complex naturally graded filiform Lie algebra is isomorphic to one of the following non isomorphic algebras:
$$\begin{array}{ll}
n_{n,1}:&\left\{ [x_i,x_1]=-[x_1,x_i]=x_{i+1}, \quad 2\leq i \leq n-1.\right.\\[3mm]
Q_{2n}:&\left\{\begin{array}{ll}
[x_i,x_1] =  -[x_1,x_i]=x_{i+1},& 2\leq i \leq 2n-2,\\[2mm]
[x_i,x_{2n+1-i}]  =  -[x_{2n+1-i},x_i]=(-1)^i\,x_{2n},& 2\leq i
\leq n.
\end{array}\right.
\end{array}$$
\end{teo}

\medskip

Let ${ L}$ be a Leibniz algebra. The ideal $I$ generated by the squares of elements of the algebra $L$, that is by the set $\{[x,x]: x\in {L}\}$, plays an important role in the theory since it determines the (possible) non-Lie character of ${L}$. From the Leibniz identity, this ideal satisfies
%\begin{equation}\label{equi}
$$[{L},I]=0.$$
%\end{equation}
Clearly, quotient algebra $L / I$ is a Lie algebra, called the {\it corresponding Lie algebra} of $L$. The map $I \times L / I \to I$, $(i,\overline{x}) \mapsto [i,x]$ endows $I$ with a structure of $L / I$-module (see \cite{huel1, huel3}).

Denote by $Q(L) = L / I \oplus I,$ then the operation $(-,-)$ defines the Leibniz algebra structure on $Q(L),$ where $$(\overline{x},\overline{y}) = \overline{[x,y]}, \quad (\overline{x},i) = [x,i],
\quad (i, \overline{x}) = 0, \quad (i,j) = 0, \qquad x, y \in L, \ i,j \in I.$$

Therefore, given a Lie algebra $G$ and a $G-$module $M,$ we can construct a Leibniz algebra $(G, M)$ by the above construction.

The main problem which occurs in this connections is a description of Leibniz algebras $L,$ such that the corresponding Leibniz algebra $Q(L)$ is isomorphic to an a priory given algebra $(G, M).$

In the present paper we restrict our attention on the case where the Lie algebra $G$ is the naturally graded filiform Lie algebra $n_{n,1}$ and the $G-$module $M$ is the Fock module or a minimal faithful module.

\subsection{Fock module over the algebra $n_{n,1}$}
First we recall the notion of Fock module over the Heisenberg algebra $H_1,$ which was introduced in \cite{Cal}. It is known that if we denote by $\overline{{x}}$ the operator associated to position and by ${\frac{\overline{\partial}}{\partial x}}$ the one associated to momentum (acting for instance on the space $V$ of differentiable functions on a single variable), then $[\overline{x},{\frac{\overline{\partial}}{\partial x}}]=\overline{1}_V$. Thus we can identify the subalgebra generated by $\overline{1},\overline{{x}}$ and
${\frac{\overline{\partial}}{\partial x}}$ with the three-dimensional Heisenberg  algebra $H_1$ whose multiplication table in the basis $\{\overline{1},\overline{x},\frac{\overline{\partial}}{\partial x}\}$ has as unique non-zero product $[\overline{x},\frac{\overline{\partial}}{\partial x}]=\overline{1}$.

For a given Heisenberg algebra $H_1$ this representation gives rise to the so-called {\it Fock module} over $H_1$,  the linear space ${\mathbb{F}}[x]$ of polinomials on $x$ ($\mathbb{F}$ denotes
the algebraically closed field with zero characteristic) with the action induced by
\begin{equation}\label{Fo}
\begin{array}{lll} ( p(x),\overline{1})& \mapsto &
p(x)\\{} ( p(x),\overline{x})& \mapsto & xp(x)\\{}
(p(x),\frac{\overline{\partial}}{\partial x})& \mapsto & \frac
{\partial}{\partial x}(p(x))
\end{array}
\end{equation}
for any $p(x) \in \mathbb{F}[x].$

Now for any filiform Lie algebra $n_{n,1}$ the define  Fock module. Algebra  $n_{n,1}$ is characterized by the existence of a basis $\{x_1, x_2, \dots, x_n\}$ (see Theorem \ref{thm2.8}) and we denote
\begin{equation}\label{base1}  \frac {\overline{\delta}}{\delta
x}= x_1,\quad \overline{x^{n-i}} = (n-i)! x_i , \quad 2 \leq i \leq n.
\end{equation}

Then the action on $n_{n,1}$ is the linear space ${\mathbb{F}}[x]$, defined by \begin{equation}\label{Fo}
\begin{array}{lll}(p(x),\overline{1})& \mapsto & p(x)\\{}(p(x),\overline{x^i}) & \mapsto & x^ip(x)\\{}
(p(x),\frac{\overline{\delta}}{\delta x}) & \mapsto & \frac {\delta(p(x))}{\delta x}\end{array}
\end{equation}

In Section 3 we are interested in studying the class of Leibniz algebras $L$ satisfying that its corresponding Lie algebra is a filiform Lie algebra $n_{n,1}$ and the $n_{n,1}$-module $I$ is isomorphic to its Fock module.

This algebra will be  called {\it filiform Fock type} Leibniz algebra and denoted by $FR(n_{n,1})$, hence we will  consider the filiform Lie algebra together with its Fock representation.

\subsection{Minimal faithful representation on the algebra $n_{n,1}$}
It is known that the minimal faithful representations of $n_{n,1}$ have dimension $n$. More precisely, if $\{x_1, x_2, \dots, x_n\}$ is a basis of $n_{n,1},$ then as a minimal faithful representations we take linear transformations with the matrices $\sum\limits_{i=1}^{n-2}E_{i,i+1},$ $ E_{1,n}, E_{2,n}, \dots, E_{n-1,n},$ on the linear space $V=\{e_1, e_2, \dots, e_{n}\},$ where $E_{i,j}$ is a matrix with  ($i, j$) -th entry equal to $1$ and others are zero.

In other words, these linear transformations have the form $$\left(\begin{matrix}0&a_1&0&\dots&0&a_2\\0&0&a_1&\dots&0&a_3\\0&0&0&\dots&0&a_{4}\\
\vdots&\vdots&\vdots&\ddots&\vdots&\vdots\\0&0&0&\dots&a_1&a_{n-1}\\0&0&0&\dots&0&a_n
\\0&0&0&\dots&0&0\end{matrix}\right)$$

The faithful representation (isomorphism) $\varphi: n_{n,1} \rightarrow End(V)^-$ is defined as follows
$$\varphi(x_1) = \sum\limits_{i=1}^{n-2}E_{i,i+1},\quad \varphi(x_i) = E_{n+1-i,n}\quad 2 \leq i \leq n.$$
i.e.,
$$\varphi([x,y])(e)= [\varphi(x), \varphi(y)](e)=\varphi(y)\big(\varphi(x) (e)\big) - \varphi(x)\big(\varphi(y)
(e)\big),$$ where $x,y \in n_{n,1},$ $\varphi(x) \in End(V),$ $e \in V.$

Now, we construct a module $V \times H_{2m+1} \rightarrow V,$ such that $$(e, x) =   \varphi(x) e.$$ Then we obtain
$$\left\{\begin{array}{ll}(e_{i}, x_1) = e_{i-1}, & 2 \leq i \leq n-1,\\[1mm]
(e_{n}, x_j) = e_{n+1-j}, & 2 \leq j \leq n,\end{array}\right.$$
the remaining products in the action being zero.

In Section 4 we deal with the category of Leibniz algebras with $n_{n,1}$ as corresponding algebra and such that the action $I \times n_{n,1} \to I$ gives rise to a minimal faithful representation of $n_{n,1}$.

\section{Classification of Filiform Fock type  Leibniz algebras}

\subsection{Classification of $FR(n_{n,1})$}

Consider a naturally graded filiform Lie algebra $n_{n,1}$ with its Fock module ${\mathbb{F}}[x]$ under the action (\ref{Fo}). Since ${\mathbb{F}}[x]$ is infinite-dimensional we obtain a family of infinite-dimensional Leibniz algebras.

\begin{teo}\label{1} The Leibniz algebra $FR(n_{n,1})$ admits a basis
$$ \{\overline{1}, \overline{x}, \overline{x^2}, \dots, \overline{x^{n-2}},\frac {\overline{\delta}}{\delta
x}, \  x^{t} \ | \ t\in \mathbb{N}\cup \{0\}\}$$
 such  that the multiplication table in this basis has the form:
$$\begin{array}{lll} [\overline{x^i},\frac {\overline{\delta}}{\delta x}]=i\overline{x^{i-1}},
& 1 \leq i \leq n-2, \\[1mm]
[\frac {\overline{\delta}}{\delta
x},\overline{x^i}]=-i\overline{x^{i-1}},  & 1 \leq i \leq n-2, \\[1mm]
[x^t,\overline{x^i}] = x^{t+i}, & 1 \leq i \leq n-2,
\\[1mm]
[x^{t},\overline{1}] = x^{t}, & [x^t,\frac
{\overline{\delta}}{\delta x}] = tx^{t-1},
 \end{array}$$ where the omitted products are equal to zero.
\end{teo}

\begin{proof}

Taking into account the action (\ref{Fo}) we conclude that $$ \{\overline{1}, \overline{x}_i,\frac {\overline{\delta}}{\delta x_i}, \  x_1^{t_1}x_2^{t_2}\dots x_k^{t_k} \ | \ t_i\in
\mathbb{N}\cup \{0\}, \ 1\leq i \leq k\}$$ is a basis of $FR(n_{n,1})$
and
$$[x^{t},\overline{1}] = x^{t}, \quad  [x^t,\frac {\overline{\delta}}{\delta x}] = tx^{t-1}, \quad
[x^t,\overline{x^i}] = x^{t+i}, \quad  1 \leq i \leq n-2.$$

Let us denote
$$[\frac {\overline{\delta}}{\delta x},\overline{1}]=q(x),\quad
[\overline{1},\overline{1}]=r(x), \quad
[\overline{x^i},\overline{1}]=p_i(x), \quad 1 \leq i \leq n-2,$$

Taking the following change of basis:
$$\frac {\overline{\delta}}{\delta x}^{\prime} = \frac
{\overline{\delta}}{\delta x} - q(x), \quad \overline{1}^{\prime}
= \overline{1} - r(x), \quad \overline{x^i}^{\prime} =
\overline{x^i} - p_i(x),\quad 1 \leq i \leq n-2,$$
 we derive
$$[\overline{x^i},\overline{1}]=0, \quad [\frac {\overline{\delta}}{\delta x},\overline{1}]=0, \quad
[\overline{1},\overline{1}]=0, \quad 1 \leq i \leq n-2.$$

We denote
$$\begin{array}{lll}[\frac {\overline{\delta}}{\delta x},\frac{\overline{\delta}}{\delta x}]=
a(x), & [\overline{x^i},\overline{x^j}]=b_{i,j}(x), & 1 \leq i, j
\leq n-2,\\[1mm]
[\overline{1}, \frac {\overline{\delta}}{\delta x}]= c(x), &
[\overline{x^i}, \frac {\overline{\delta}}{\delta x}]=
i\overline{x^{i-1}}+ d_{i}(x), & 1 \leq i \leq n-2,\\[1mm]
[\overline{1}, \overline{x^i}]= g_{i}(x),& [\frac
{\overline{\delta}}{\delta x},\overline{x^i}]= -
i\overline{x^{i-1}} +  h_{i}(x), & 1 \leq i \leq n-2.\end{array}$$

Consider the Leibniz identity
$$[[\frac {\overline{\delta}}{\delta x},
\frac{\overline{\delta}}{\delta x}],\overline{1}] = [\frac
{\overline{\delta}}{\delta x},[\frac {\overline{\delta}}{\delta
x},\overline{1}]] + [[\frac {\overline{\delta}}{\delta
x},\overline{1}],\frac {\overline{\delta}}{\delta x}] =0,$$

On the other hand,
$$[[\frac {\overline{\delta}}{\delta x},
\frac{\overline{\delta}}{\delta x}],\overline{1}] =
[a(x),\overline{1}] =a(x),$$ which implies $a(x)=0.$

Similarly, from the Leibniz identities
$$
\begin{array}{lllllllllll}b_{i,j}(x)&=& [b_{i,j}(x), \overline{1}] &=& [[\overline{x^i},\overline{x^j}],\overline{1}]
&=& [\overline{x^i},[\overline{x^j},\overline{1}]] &+&
[[\overline{x^i},\overline{1}],\overline{x^j}] &=&0,\\[1mm]
c(x)&=&[c(x),\overline{1} ] &=& [[\overline{1},
\frac{\overline{\delta}}{\delta x_i}],\overline{1}]
 &=& [\overline{1},[\frac {\overline{\delta}}{\delta
x_i},\overline{1}]] & + & [[\overline{1},\overline{1}],\frac
{\overline{\delta}}{\delta x_i}] &=&0,\\[1mm]
d_{i}(x)&=&[i\overline{x^{i-1}} + d_{i}(x),\overline{1}]
&=&[[\overline{x^i}, \frac {\overline{\delta}}{\delta
x}],\overline{1}] &=& [\overline{x^i},[\frac
{\overline{\delta}}{\delta x},\overline{1}]] &+&
[[\overline{x^i},\overline{1}],\frac {\overline{\delta}}{\delta
x}] & = &0,\\[1mm]
g_{i}(x)& =& [g_{i}(x),\overline{1}] &=& [[\overline{1},
\overline{x^i}],\overline{1}] & =&
[\overline{1},[\overline{x^i},\overline{1}]] & +&
[[\overline{1},\overline{1}],\overline{x^i}] & =&0,\\[1mm]
h_{i}(x)& =&[-i\overline{x^{i-1}} + h_{i}(x),\overline{1}] &
=&[[\frac {\overline{\delta}}{\delta x},
\overline{x^i}],\overline{1}] & =& [\frac
{\overline{\delta}}{\delta x},[\overline{x^i},\overline{1}]] & +&
[[\frac {\overline{\delta}}{\delta
x},\overline{1}],\overline{x^i}] & =&0,\end{array}$$

we obtain
$$
\begin{array}{llll}& c(x)=0, & b_{i,j}=0, & 1 \leq i,j \leq n-2,\\[1mm]
d_{i}(x)=0, &  g_{i}(x)=0, & h_{i}(x)=0, & 1 \leq i \leq
n-2.\end{array}$$
\end{proof}

\subsection{Classification of generalized Filiform Fock type Leibniz algebras}

In this subsection  we are focused in classifying of the class of
(infinite-dimensional) Leibniz algebras $L$ such that their corresponding Lie algebras are
finite direct sums of filiform Lie algebras $n_{n_1,1}\oplus
n_{n_2,1} \oplus \dots \oplus n_{n_s,1}$ and  that their actions on
$I $ are induced by Fock representations.

Since each algebra  $n_{n_i,1}$ has a standard basis $\{x_{i,1}, x_{i,2}, \dots, x_{i,n}\}$ we put \begin{equation}\label{base1}
  \frac {\overline{\delta}}{\delta
x_i} = x_{i,1},\qquad \overline{x_i^{n-j}}  = (n_i-j)! x_{i,j},
\quad 2 \leq j \leq n_i.
\end{equation}

For the algebra $n_{n_1,1}\oplus n_{n_2,1} \oplus \dots \oplus
n_{n_s,1}$ the {\it Fock module} on $n_{n_1,1}\oplus n_{n_2,1}
\oplus \dots \oplus n_{n_s,1}$ is the linear space
${\mathbb{F}}[x_1, x_2, \dots, x_s]$  with the action induced by
$$\begin{array}{lcl}(p(x_1, x_2, \dots, x_s),\overline{1_i}) & \mapsto &
p(x_1, x_2, \dots, x_s), \quad 1 \leq i \leq s,\\[1mm]
( p(x_1, x_2, \dots, x_s),\overline{x_i^j}) &\mapsto & x_i^j
p(x_1, x_2, \dots, x_s),\\[1mm]
(p(x_1, x_2, \dots, x_s),\frac {\overline{\delta}}{\delta x_i}) &
\mapsto & \frac {\delta(p(x_1, x_2, \dots, x_s))}{\delta
x_i},\end{array}$$ for any $p(x_1, x_2, \dots, x_s) \in
\mathbb{F}[x_1, x_2, \dots, x_s].$

We denote
$$\begin{array}{lll}
[\overline{x_i^j},\overline{1_k}]=a_{i,k}^j(x_1, x_2, \dots, x_s),
& 1 \leq i,k \leq s, & 1 \leq j \leq n_i-2,\\[1mm] [\frac
{\overline{\delta}}{\delta x_i},\overline{1_j}]=b_{i,j}(x_1, x_2,
\dots, x_s), & 1 \leq i, j \leq
s,\\[1mm]
[\overline{1_i},\overline{1_j}]=c_{i,j}(x_1, x_2, \dots, x_s), & 1
\leq i, j \leq s.\end{array}$$

Taking the change of basis
$$\begin{array}{ll}
\overline{x_i^j}^{\prime} = \overline{x_i^j} - a_{i,i}^k(x_1, x_2, \dots, x_s),& 1 \leq i \leq s,
\ 1 \leq j \leq n_i-2,\\[3mm]
\frac {\overline{\delta}}{\delta x_i}^{\prime} = \frac {\overline{\delta}}{\delta x_i} -
 b_{i,i}(x_1, x_2, \dots, x_s),&\\[3mm]
\overline{1_i}^{\prime} = \overline{1_i} - c_{i,i}(x_1, x_2, \dots, x_s),&
\end{array}$$
we derive
$$[\overline{x_i^j},\overline{1_i}]=0, \quad [\frac {\overline{\delta}}{\delta x_i},\overline{1_i}]=0, \quad
[\overline{1_i},\overline{1_i}]=0, \quad 1 \leq i \leq s, \ 1 \leq
j \leq n_i-2.$$

Let us introduce notations:

$$\begin{array}{llll}
[\overline{x_i^j},\overline{x_k^t}]=d_{i,k}^{j,t}(x_1, x_2, \dots,
x_s), & 1 \leq i, k \leq s, & 1 \leq j \leq n_i-2, & 1 \leq t \leq
n_k-2,\\[2mm] [\frac {\overline{\delta}}{\delta x_i},\frac{\overline{\delta}}{\delta x_j}]=
e_{i,j}(x_1, x_2, \dots, x_s), & 1 \leq i,j \leq s,\\[2mm]
[\overline{x_i^j}, \frac {\overline{\delta}}{\delta x_i}]=
i\overline{x_i^{j-1}}+ f^j_{i,i}(x_1, x_2, \dots, x_s), & 1 \leq i
\leq s, & 1 \leq j \leq n_i-2,\\[2mm]
[\overline{x_i^j}, \frac {\overline{\delta}}{\delta x_k}]=
f^j_{i,k}(x_1, x_2, \dots, x_s), & 1 \leq i, k \leq s, & 1 \leq j
\leq n_i-2, & i \neq k,\\[2mm] [\frac {\overline{\delta}}{\delta
x_i},\overline{x_i^j}]= - i\overline{x_i^{j-1}} +  g^j_{i,i}(x_1,
x_2, \dots, x_s), & 1 \leq i \leq s, & 1 \leq j \leq n_i-2,\\[2mm]
[\frac {\overline{\delta}}{\delta x_k},\overline{x_i^j}]=
g^j_{k,i}(x_1, x_2, \dots, x_s), & 1 \leq i \leq s, & 1 \leq j
\leq n_i-2, & i \neq k,\\[2mm]
[\overline{1_k}, \overline{x_i^j}]= h_{k,i}^j(x_1, x_2, \dots,
x_s), & 1 \leq i,k \leq s, & 1 \leq j \leq n_i-2,\\[2mm]
[\overline{1_i}, \frac {\overline{\delta}}{\delta x_j}]=
q_{i,j}(x_1, x_2, \dots, x_s), & 1 \leq i,k \leq s.
\end{array}$$

Consider the Leibniz identity
$$[[\overline{x_i^j},\overline{1_k}],\overline{1_i}] =
[\overline{x_i^j},[\overline{1_k},\overline{1_i}]] +
[[\overline{x_i^j},\overline{1_i}],\overline{1_k}] =0.$$

On the other hand
$$[[\overline{x_i^j},\overline{1_k}],\overline{1_i}] =
[a_{i,k}^j(x_1, x_2, \dots, x_s),\overline{1_i}] = a_{i,k}^j(x_1,
x_2, \dots, x_s)$$ which implies
$$a_{i,k}^j(x_1, x_2, \dots, x_s)=0, \quad 1 \leq i,k \leq s, \ 1 \leq j \leq n_i-2.$$

Similarly from the Leibniz identities
$$\begin{array}{l}
b_{i,j}(x_1, x_2, \dots, x_s)=[[\frac {\overline{\delta}}{\delta x_i},\overline{1_j}],\overline{1_i}] = [\frac
{\overline{\delta}}{\delta x_i},[\overline{1_j},\overline{1_i}]] +
[[\frac {\overline{\delta}}{\delta
x_i},\overline{1_i}],\overline{1_j}] =0,\\[3mm]
c_{i,j}(x_1, x_2, \dots, x_s)=[[\overline{1_i},\overline{1_j}],\overline{1_i}] =
 [\overline{1_i},[\overline{1_j},\overline{1_i}]] +
[[\overline{1_i},\overline{1_i}],\overline{1_j}] =0,
\end{array}$$ we obtain
$$b_{i,j}(x_1, x_2, \dots, x_s)=0, \quad c_{i,j}(x_1, x_2, \dots,
 x_s)=0, \quad 1 \leq i,j \leq s.$$

In a similar way, from the Leibniz identities
$$\begin{array}{l}
d_{i,k}^{j,t}(x_1, x_2,
\dots, x_s)=[[\overline{x_i^j},\overline{x_k^t}],\overline{1_i}] =
[\overline{x_i^j},[\overline{x_k^t},\overline{1_i}]] +
[[\overline{x_i^j},\overline{1_i}],\overline{x_k^t}] =0,\\[3mm]
e_{i,j}(x_1, x_2, \dots, x_s)=[[\frac {\overline{\delta}}{\delta x_i},\frac{\overline{\delta}}{\delta x_j}],
\overline{1_i}] = [\frac {\overline{\delta}}{\delta
x_i},[\frac{\overline{\delta}}{\delta x_j},\overline{1_i}]] +
[[\frac {\overline{\delta}}{\delta
x_i},\overline{1_i}],\frac{\overline{\delta}}{\delta x_j}] =0,\\[3mm]
f^j_{i,i}(x_1, x_2, \dots, x_s)=[i\overline{x_i^{j-1}}+ f^j_{i,i}(x_1, x_2, \dots, x_s),\overline{1_i}]=
[[\overline{x_i^j}, \frac {\overline{\delta}}{\delta x_i}],
\overline{1_i}] = [\overline{x_i^j},[\frac
{\overline{\delta}}{\delta x_i},\overline{1_i}]] +
[[\overline{x_i^j},\overline{1_i}],\frac
{\overline{\delta}}{\delta x_i}] =0,\\[3mm]
f^j_{i,k}(x_1, x_2, \dots, x_s)=
[[\overline{x_i^j}, \frac {\overline{\delta}}{\delta x_k}],
\overline{1_i}] = [\overline{x_i^j},[\frac
{\overline{\delta}}{\delta x_k},\overline{1_i}]] +
[[\overline{x_i^j},\overline{1_i}],\frac
{\overline{\delta}}{\delta x_k}] =0,\\[3mm]
g^j_{i,i}(x_1,
x_2, \dots, x_s)= [- i\overline{x_i^{j-1}} +  g^j_{i,i}(x_1, x_2,
\dots, x_s), \overline{1_i}] = [[\frac {\overline{\delta}}{\delta
x_i},\overline{x_i^j}], \overline{1_i}] = [\frac
{\overline{\delta}}{\delta x_i},[\overline{x_i^j},\overline{1_i}]]
+ [[\frac {\overline{\delta}}{\delta
x_i},\overline{1_i}],\overline{x_i^j}] =0,\\[3mm]
g^j_{k,i}(x_1,
x_2, \dots, x_s)= [g^j_{k,i}(x_1, x_2, \dots, x_s),
\overline{1_i}] = [[\frac {\overline{\delta}}{\delta
x_k},\overline{x_i^j}], \overline{1_k}] = [\frac
{\overline{\delta}}{\delta x_k},[\overline{x_i^j},\overline{1_k}]]
+ [[\frac {\overline{\delta}}{\delta
x_k},\overline{1_k}],\overline{x_i^j}] =0,\\[3mm]
h_{k,i}^j(x_1, x_2, \dots, x_s)= [h_{k,i}^j(x_1, x_2, \dots, x_s), \overline{1_k}] =
[[\overline{1_k}, \overline{x_i^j}], \overline{1_k}] =
[\overline{1_k},[\overline{x_i^j},\overline{1_k}]] +
[[\overline{1_k},\overline{1_k}],\overline{x_i^j}] =0,\\[3mm]
q_{i,j}(x_1, x_2, \dots, x_s)= [q_{i,j}(x_1, x_2, \dots, x_s), \overline{1_i}] =
[[\overline{1_i}, \frac {\overline{\delta}}{\delta x_j}],
\overline{1_i}] = [\overline{1_i},[\frac
{\overline{\delta}}{\delta x_j},\overline{1_i}]] +
[[\overline{1_i},\overline{1_i}],\frac {\overline{\delta}}{\delta
x_j}] =0,
\end{array}$$

we derive
$$\begin{array}{llll}
d_{i,k}^{j,t}(x_1, x_2, \dots, x_s) & 1 \leq i,k \leq s,& 1 \leq j
\leq n_i-2, & 1 \leq t \leq
n_k-2,  \\[1mm]
e_{i,j}(x_1, x_2, \dots, x_s)=0, & 1 \leq i,j \leq s,& \\[1mm]
f^j_{i,k}(x_1, x_2, \dots, x_s)=0, & 1 \leq i, k \leq s, & 1 \leq
j \leq n_i-2,\\[1mm]
g^j_{k,i}(x_1, x_2, \dots, x_s)=0, & 1 \leq i \leq s, & 1 \leq j
\leq n_i-2,\\[1mm]
h_{k,i}^j(x_1, x_2, \dots, x_s)=0, & 1 \leq i,k \leq s, & 1 \leq j
\leq n_i-2,\\[1mm]
q_{i,j}(x_1, x_2, \dots, x_s)=0, & 1 \leq i,k \leq s. &
\end{array}$$

Therefore, we have proved
\begin{teo} The above Leibniz algebra denoted by $FR(n_{n_1,1}\oplus n_{n_2,1} \oplus \dots \oplus n_{n_s,1})$
admits a basis
$$
 \{\overline{1_i}, \overline{x}_i^j,\frac {\overline{\delta}}{\delta
x_i}, \  x_1^{t_1}x_2^{t_2}\dots x_k^{t_k} \ | \ t_i\in
\mathbb{N}\cup \{0\}, \ 1\leq i \leq s, 1 \leq j \leq n_i\}
$$
in such  that the multiplication table in this basis has the
form:
$$\begin{array}{l} [\overline{x_i},\frac {\overline{\delta}}{\delta x_i}]=\overline{1},
\hspace{0.8cm} [\frac {\overline{\delta}}{\delta
x_i},\overline{x_i}]=-\overline{1},  \hspace{0.8cm} 1 \leq i \leq
k, \\[1mm]
[x_1^{t_1}x_2^{t_2}\dots x_k^{t_k},\overline{1}] =
x_1^{t_1}x_2^{t_2}\dots x_k^{t_k},\\[1mm]
[x_1^{t_1}x_2^{t_2}\dots x_k^{t_k},\overline{x}_i^j] =
x_1^{t_1}\dots x_{i-1}^{t_{i-1}}x_i^{t_i+j} x_{i+1}^{t_{i+1}}\dots
x_k^{t_k},\\[1mm] [x_1^{t_1}x_2^{t_2}\dots x_k^{t_k},\frac
{\overline{\delta}}{\delta x_i}] = t_i x_1^{t_1}\dots
x_{i-1}^{t_{i-1}} x_i^{t_i-1}x_{i+1}^{t_{i+1}}\dots
x_k^{t_k}.\end{array}$$ where the omitted products are equal to
zero.
\end{teo}

\section{Leibniz algebras associated with minimal faithful representation of $n_{n,1}$}

In this section we are going to study  the Leibniz algebras $L$ such that $L/ I \cong n_{n,1}$ and the
$n_{n,1}$-module $I$ is the  minimal faithful representation. In this case we have that
 ${\rm dim} L = 2n$ and $\{ x_1, x_2, \dots, x_n, e_1, e_2, \dots, e_{n}\}$ is  a basis of $L.$ We also have
\begin{equation}\label{eq1}\left\{\begin{array}{ll}
[e_{i}, x_1] = e_{i-1}, & 2 \leq i \leq n-1,\\[1mm]
[e_{n},x_j] = e_{n+1-j}, & 2 \leq j \leq n.
\end{array}\right.\end{equation}

Further we should define the multiplications $[x_i, x_j]$ for $1
\leq i,j \leq n.$ We put
\begin{equation}\label{eq_table}[x_i,x_j]=\left\{\begin{array}{ll} x_{i+1} +
\sum\limits_{k=1}^{n} \alpha_{i,1}^k e_k, & j=1, \  2
\leq i \leq n-1, \\[1mm]
-x_{j+1} + \sum\limits_{k=1}^{n} \alpha_{1,j}^k e_k, & i=1, \
2 \leq j \leq n-1,\\[1mm] \sum\limits_{k=1}^{n} \alpha_{i,j}^k e_k, &
otherwise.\end{array}\right.\end{equation}

In the following Lemma we define the multiplications in the case
of $i=1$ or $j=1$ and $j=n.$

\begin{lemma}\label{lemma1} There exists a basis $\{ x_1, x_2, \dots, x_n, e_1, e_2, \dots,
e_{n}\}$ of $L$ such that
\begin{equation}\label{eq4.1}\left\{ \begin{array}{lll}
[x_1,x_1]= \alpha_{1} e_{n-1}+\alpha_{2} e_{n}, & [x_1,x_j]=-x_{j+1}, &   2 \leq  j \leq n-1,\\[1mm]
[x_2,x_1]=x_3+\alpha_{3} e_{n}, & [x_{i},x_{1}] = x_{i+1}-
\alpha_2
e_{n+2-i}, &  3 \leq i \leq n-1, \\[1mm]
[x_{1},x_{n}] = \alpha_4 e_1 + \alpha_2 e_{2}, &
[x_{n},x_{1}] = - \alpha_4e_1 - 2\alpha_2e_{2}, \\[1mm]
[x_{2},x_{n}] = \alpha_5 e_1 + \alpha_3e_{2}, & [x_{i},x_{n}] =0,
& 3 \leq i \leq n.
\end{array}\right.
\end{equation}

\end{lemma}

\begin{proof}

In the multiplication \eqref{eq_table} taking the transformation
of basis
$$\begin{array}{ll}
x_1' = x_1 - \sum\limits_{k=1}^{n-2}\alpha_{1,1}^{k}e_{k+1} -
(\alpha_{2,1}^{n-1}+\alpha_{1,2}^{n-1})e_{n},& x_2' = x_2 -
\sum\limits_{k=1}^{n-2}(\alpha_{2,1}^{k}+\alpha_{1,2}^{k})e_{k+1},\\[3mm]
x_j' = x_j - \sum\limits_{k=1}^{n}\alpha_{1,j-1}^{k}e_{k} +(\alpha_{2,1}^{n-1}+\alpha_{1,2}^{n-1})e_{n+2-j},
& 3 \leq j \leq n,
\end{array}$$ we obtain
%$$\begin{cases}[x_1,x_1]= \alpha_{1,1}^{n-1} e_{n-1}+\alpha_{1,1}^{n} e_{n},
%& [x_1,x_j]=-x_{j+1}, \quad 2 \leq  j \leq n-1,\\[1mm]
%[x_2,x_2]=\sum\limits_{k=1}^{n-2}\alpha_{2,2}^{k} e_{k}+
%\alpha_{2,2}^{n} e_{n}, & [x_2,x_1]=x_3+\alpha_{2,1}^{n}
%e_{n}.\end{cases}$$
$[x_1,x_1]= \alpha_{1,1}^{n-1} e_{n-1}+\alpha_{1,1}^{n} e_{n},
\quad [x_2,x_1]=x_3+\alpha_{2,1}^{n} e_{n}, \quad
[x_1,x_j]=-x_{j+1}, \quad 2 \leq  j \leq n-1.$

Using the Leibniz identity we derive
$$\begin{array}{lll}
[x_{3},x_{1}] &=& -[[x_1,x_2],
x_{1}] = -[x_1,[x_2, x_{1}]] - [[x_1,x_1], x_{2}] = \\[1mm]
&=&[-x_1,
x_3+\alpha_{2,1}^{n-1} e_{n-1}+\alpha_{2,1}^{n} e_{n}] -
[\alpha_{1,1}^{n-1} e_{n-1}+\alpha_{1,1}^{n} e_{n}, x_2] =x_4 -
\alpha_{1,1}^{n} e_{n-1}.\\[1mm]
\end{array}$$

From the Leibniz identity, $[[x_1, x_{i}], x_1] = [x_1,[x_i,
x_{1}]] + [[x_1,x_1], x_{i}] $ recurrently we obtain
$$\begin{cases}[x_{i},x_{1}] = x_{i+1}-
\alpha_{1,1}^{n} e_{n+2-i}, & 3 \leq i \leq n-1, \\[1mm]
[x_{n},x_{1}] =  -\alpha_{1,1}^{n} e_{2} -
\sum\limits_{k=1}^{n}\alpha_{1,n}^{k}e_k.\end{cases}$$

On the other hand, from
$$0= [x_1,[x_n, x_{1}]] = [[x_1,x_n], x_{1}] - [[x_1,x_1], x_{n}]
= \sum\limits_{k=2}^{n-1}\alpha_{1,n}^{k}e_{k-1}-
\alpha_{1,1}^ne_1,$$ we get
$$\alpha_{1,n}^{2} = \alpha_{1,1}^{n}, \qquad
\alpha_{1,n}^k = 0, \quad 3 \leq k \leq n-1.$$

Now, we consider the Leibniz identity
$$\begin{array}{lll}
0&=& [x_1,[x_n, x_{j}]] = [[x_1,x_n], x_{j}] - [[x_1,x_j], x_{n}]
=\\[1mm]
&=&[\alpha_{1,n}^{1}e_1 + \alpha_{1,1}^{n} e_{2}  +
\alpha_{1,n}^{n}e_n, x_j] + [x_{j+1},x_n] =
\alpha_{1,n}^{n}e_{n+1-j} + [x_{j+1},x_n].
\end{array}$$

Hence, we have
$$ [x_{j+1},x_n] = - \alpha_{1,n}^{n}e_{n+1-j}, \quad 2 \leq j \leq n-1.$$

On the other hand, from the equalities
$$\begin{array}{lll}
0&=& [x_2,[x_1, x_{n}]] = [[x_2, x_1], x_{n}] - [[x_2,x_n], x_{1}]
= \\[1mm]
&=&[x_3 + \alpha_{2,1}^{n}e_n, x_n] -
[\sum\limits_{k=1}^{n-1}\alpha_{2,n}^{k}e_k,x_1] =
-\alpha_{1,n}^{n}e_{n-1} + \alpha_{2,1}^{n}e_1 -
\sum\limits_{k=1}^{n-2}\alpha_{2,n}^{k+1}e_k,
\end{array}$$
we obtain
$$\alpha_{2,1}^{n} = \alpha_{2,n}^{2}, \quad \alpha_{1,n}^{n}=0, \quad  \alpha_{2,n}^{k} =0, \quad 3 \leq k \leq n-1.$$

\end{proof}

Put
$$\begin{array}{ll}
Q_{0,1}=1, \qquad \quad   Q_{0,k}=\frac 1 2, \quad k \geq 2, & Q_{1,k}=\frac {k+1} 2,\\[3mm]
Q_{m,k}=\displaystyle\frac {k(k+1) \dots (k+m-2)(k+2m-1)} {2 (m!)}, & m \geq 2.
\end{array}$$

It is not difficult to check that
\begin{equation}\label{Q_m}Q_{m,k} = Q_{m,k-1} +
Q_{m-1,k}.\end{equation}

Now we will define the products $[x_i,x_j]$ for $i+ j \leq n+1.$
\begin{lemma}\label{lemma2} We have
\begin{equation}\label{eq4.2}
\small\left\{\begin{array}{lll}
[x_2,x_2]&=\sum\limits_{k=1}^{n-2}\beta_{k}e_k, \\[2mm]
[x_{i+1},x_i]&=\sum\limits_{k=1}^{n-1}\gamma_{i,k}e_k, & 2 \leq
i \leq \lfloor\frac{n} 2\rfloor, \\[1mm]
\\[1mm]
 [x_{i}, x_{i+j}] &=
\sum\limits_{s=0}^{\lfloor\frac {j+1}2\rfloor}
(-1)^{s}Q_{s,j+2-2s}\sum\limits_{k=1}^{n-2-j+2s}\gamma_{i+s-1,j+1-2s+k}e_k,
& 0 \leq j \leq n -5, \quad 3 \leq i \leq \lfloor\frac {n+1-j}
2\rfloor,\\[1mm]
[x_{2}, x_{j}] &= - (j-2)\alpha_3e_{n+2-j} +
\sum\limits_{k=1}^{n-j}\beta_{j-2+k}e_k+ \\[1mm] &+ \sum\limits_{s=2}^{\lfloor\frac
{j+1}2\rfloor}
(-1)^{s+1}Q_{s-1,j+2-2s}\sum\limits_{k=1}^{n-2-j+2s}\gamma_{s,j+1-2s+k}e_k,
 & 3 \leq j \leq n-1, \end{array}\right.\end{equation}
where $\lfloor a\rfloor$ is the integer part of $a.$
\end{lemma}

\begin{proof}
Taking into account the notation \eqref{eq_table}, from the
Leibniz identity
$$\begin{array}{lll}
0&=& [x_i,[x_j, x_{k}]] = [[x_i,x_j], x_{k}] - [[x_i,x_k], x_{j}]
=\left [\sum\limits_{t=1}^n \alpha_{i,j}^t e_t, x_k\right ] +
\left [\sum\limits_{t=1}^n \alpha_{i,k}^t e_t, x_i\right ]
=\\[3mm]
&=&\alpha_{i,j}^ne_{n+2-k} - \alpha_{i,k}^ne_{n+2-j}, \quad 2 \leq
i,j,k(j \neq k) \leq n,
\end{array}$$ we get
$$\alpha_{i,j}^n =0, \quad 2 \leq i,j \leq n.$$

From the Leibniz identities for the triples of elements $[x_1,[x_i,
x_{j}]], [x_i,[x_1, x_{j}]],$ we derive the following relations
\begin{equation}\label{eq4.3}\left\{\begin{array}{ll}
[x_{i+1}, x_j] = [x_{j+1}, x_i], & 2 \leq i,j \leq n-1,  \\[1mm] [x_{2},
x_{j+1}] + [x_{3}, x_j] = - \alpha_3 e_{n+1-j} +
[[x_2,x_j], x_1], & 2 \leq j \leq n-1, \\[1mm] [x_{i}, x_{j+1}] +
[x_{i+1}, x_j] = [[x_i,x_j], x_1], & 3 \leq i \leq n-1 , \ 2 \leq
j \leq n-1, \\[1mm]
[x_n, x_{j+1}] = [[x_n, x_j],x_1] & 2 \leq j \leq n-1.
\end{array}\right.\end{equation}

From the first equality in \eqref{eq4.3} it is easy to see that it
is sufficient to define the multiplications $[x_i, x_j]$ for $j \geq
i-1.$  Put,
$$[x_2,x_2]=\sum\limits_{k=1}^{n-1}\beta_{k}e_k, \quad [x_{i+1},x_i]=\sum\limits_{k=1}^{n-1}\gamma_{i,k}e_k, \quad 2
\leq i \leq \lfloor\frac{n} 2\rfloor.$$

Applying, if necessary, the change of basis $x_2' = x_2- \beta_{n-1}e_n$ we can suppose that
$\beta_{n-1}=0$ and we will express other product by means of the structure
constants $\beta_i$ and $\gamma_{i,j}.$

First we will proof the third equation from \eqref{eq4.2} by
induction on $j.$ From the relations \eqref{eq4.3} we get $[x_{i},
x_{i}] = [x_{i+1}, x_{i-1}]$ and $[x_{i}, x_{i}] + [x_{i+1},
x_{i-1}] =[[x_{i}, x_{i-1}], x_1]$ which imply the assertion of
Lemma for $j=0,$ i.e.,
$$[x_{i}, x_{i}] = \frac 1 2 [[x_{i}, x_{i-1}], x_1]=\frac 1 2 \sum\limits_{k=1}^{n-2}\gamma_{i-1,k+1}e_k, \quad 3 \leq i \leq
\lfloor\frac{n+1} 2\rfloor.$$

Then from the relations (1) we obtain $[x_{i}, x_{i+1}] +
[x_{i+1}, x_{i}] = [[x_i, x_i], x_1].$ Using the assertion of
Lemma for $j=0,$ we get
$$[x_{i}, x_{i+1}] = \frac
1 2 [[[x_{i}, x_{i-1}], x_1],x_1] - [x_{i+1},x_i] = \frac 1 2
\sum\limits_{k=1}^{n-3}\gamma_{i-1,k+2}e_k -
\sum\limits_{k=1}^{n-1}\gamma_{i,k}e_k, \quad 3 \leq i \leq
\lfloor\frac{n} 2\rfloor.$$ Hence the assertion of the Lemma is true
for $j=1.$

Let us suppose that the assertion of the Lemma is true for indices less or equal to
$j$ and we will prove it for $j+1.$

From the relations \eqref{eq4.3} we obtain $[x_{i}, x_{i+j+1}] +
[x_{i+1}, x_{i+j}] = [[x_i, x_{i+j}], x_1].$ Using the assumption
of the induction we get
$$\begin{array}{l}
[x_{i}, x_{i+j+1}] = [[x_i, x_{i+j}],
x_1] - [x_{i+1}, x_{i+j}] =\\[2mm]
\qquad =\sum\limits_{s=0}^{\lfloor\frac
{j+1}2\rfloor}(-1)^{s}Q_{s,j+2-2s}\sum\limits_{k=1}^{n-3-j+2s}\gamma_{i+s-1,j+2-2s+k}e_k
 -\sum\limits_{s=0}^{\lfloor\frac {j}2\rfloor} (-1)^{s}Q_{s,j+1-2s}
\sum\limits_{k=1}^{n-1-j+2s}\gamma_{i+s,j-2s+k}e_k=\\[3mm]
\qquad =Q_{0,j+2}\sum\limits_{k=1}^{n-3-j}\gamma_{i-1,j+2+k}e_k + \sum\limits_{s=1}^{\lfloor\frac {j+1}2\rfloor}
(-1)^{s}Q_{s,j+2-2s}
\sum\limits_{k=1}^{n-3-j+2s}\gamma_{i+s-1,j+2-2s+k}e_k-\\[3mm]
\qquad -\sum\limits_{s=1}^{\lfloor\frac {j}2\rfloor+1}
(-1)^{s-1}Q_{s-1,j+3-2s}
\sum\limits_{k=1}^{n-3-j+2s}\gamma_{i+s-1,j+2-2s+k}e_k.
\end{array}$$

If $j$ is odd then $\lfloor\frac {j+1}2\rfloor = \lfloor\frac {j}2\rfloor
+1= \lfloor\frac {j+2}2\rfloor $ and using the equality \eqref{Q_m} we
get
$$\begin{array}{lll}
[x_{i}, x_{i+j+1}]& =&  Q_{0,j+2}\sum\limits_{k=1}^{n-3-j}\gamma_{i-1,j+2+k}e_k
+ \sum\limits_{s=1}^{\lfloor\frac {j+1}2\rfloor} (-1)^{s}Q_{s,j+3-2s}
\sum\limits_{k=1}^{n-3-j+2s}\gamma_{i+s-1,j+2-2s+k}e_k  =\\[3mm]
&=& \sum\limits_{s=0}^{\lfloor\frac {j+2}2\rfloor}
(-1)^{s}Q_{s,j+3-2s}
\sum\limits_{k=1}^{n-3-j+2s}\gamma_{i+s-1,j+2-2s+k}e_k.
\end{array}$$

If $j$ is even then $\lfloor\frac {j+1}2\rfloor = \lfloor\frac
{j}2\rfloor$ and we get
$$\begin{array}{lll}
[x_{i}, x_{i+j+1}] &=& Q_{0,j+2}\sum\limits_{k=1}^{n-3-j}\gamma_{i-1,j+2+k}e_k
+ \sum\limits_{s=1}^{\lfloor\frac {j}2\rfloor}
(-1)^{s}(Q_{s,j+3-2s})\sum\limits_{k=1}^{n-3-j+2s}\gamma_{i+s-1,j+2-2s+k}e_k
+\\[3mm]
&+& (-1)^{\lfloor\frac{j+2}2\rfloor}
\sum\limits_{k=1}^{n-1}\gamma_{i+\lfloor\frac {j}2 \rfloor,k}e_k =
\sum\limits_{s=0}^{\lfloor\frac {j+2}2\rfloor} (-1)^{s}Q_{s,j+3-2s}
\sum\limits_{k=1}^{n-3-j+2s}\gamma_{i+s-1,j+2-2s+k}e_k.
\end{array}$$

The products $[x_2, x_j]$ are also obtained by
induction on $j,$ using the equality  \eqref{eq4.3} and the
multiplication $[x_3, x_{j-1}].$
\end{proof}

In the following lemma we define the products $[x_i,x_j]$ for $i+j
\geq n+2.$

\begin{lemma}\label{lemma3} We have
\begin{equation}\label{eq4.4}\small\left\{\begin{array}{lll} [x_{i},x_{n+2-i}] &= (-1)^{i} \alpha_5 e_1 +
(-1)^{i}(n-5)\alpha_3 e_2+(-1)^{i+1}\beta_{n-2}e_1+\\[3mm] &+
\sum\limits_{s=2}^{i-2}
(-1)^{s+i}\sum\limits_{t=1}^sQ_{s-t,n+1-2s}\sum\limits_{k=1}^{2s-2}\gamma_{s,n+1-2s+k}e_k +\\[3mm]
&+ \sum\limits_{s=i-1}^{\lfloor\frac n 2\rfloor}(-1)^{s+i}\sum\limits_{t=1}^{i-2}Q_{s-t,n+1-2s}\sum\limits_{k=1}^{2s-2}\gamma_{s,n+1-2s+k}e_k, \\[3mm]
[x_{i},x_{n+3-i}] &= (-1)^{i+1} (i-3)(n-5)\alpha_3 e_1+ \\[3mm]
&+ \sum\limits_{s=2}^{i-3}
(-1)^{s+i+1}\sum\limits_{t=1}^s(i-2-t)Q_{s-t,n+1-2s}\sum\limits_{k=1}^{2s-3}\gamma_{s,n+2-2s+k}e_k+\\[3mm]
&+ \sum\limits_{s=i-2}^{\lfloor\frac n 2\rfloor}(-1)^{s+i+1}\sum\limits_{t=1}^{i-3}(i-2-t)Q_{s-t,n+1-2s}\sum\limits_{k=1}^{2s-3}\gamma_{s,n+2-2s+k}e_k,\\[3mm]
[x_{i},x_{n+p-i}] & = \sum\limits_{s=\lfloor\frac p 2\rfloor +1}^{i-p}
(-1)^{s+i+p}\sum\limits_{t=1}^s\Big(\begin{array}{c} i-2-t\\
p-2\end{array}\Big)
Q_{s-t,n+1-2s}\sum\limits_{k=1}^{2s-p}\gamma_{s,n+p-1-2s+k}e_k +\\[1mm]
&+ \sum\limits_{s=\max\{i-p+1; \lfloor\frac p 2\rfloor +1\}}^{\lfloor\frac n
2\rfloor}(-1)^{s+i+p}
\sum\limits_{t=1}^{i-p}\Big(\begin{array}{c} i-2-t\\
p-2\end{array}\Big)Q_{s-t,n+1-2s}\sum\limits_{k=1}^{2s-p}\gamma_{s,n+p-1-2s+k}e_k,\end{array}\right.\end{equation}
where $4 \leq p \leq n-1,$\  $p+1 \leq i \leq \lfloor\frac {n+p+1} 2\rfloor.$
\end{lemma}

\begin{proof}
First we will find the products $[x_{i},x_{n+2-i}].$

According to Lemma \ref{lemma1} we have $[x_{2},x_n] = \alpha_5
e_1 + \alpha_3e_{2}$ and using Lemma \ref{lemma2} from relations
\eqref{eq4.3} we obtain
$$\begin{array}{lll}
[x_{3},x_{n-1}] &=& - \alpha_3e_{2}- [x_2,x_{n}] + [[x_2,x_{n-1}], x_1]
= -\alpha_5 e_1 - (n-5)\alpha_3 e_2 + \\[3mm]
&+&\sum\limits_{s=2}^{\lfloor\frac n 2\rfloor}
(-1)^{s+1}Q_{s-1,n+1-2s}\sum\limits_{k=1}^{2s-2}
\gamma_{s,n+1-2s+k}e_k.
\end{array}$$

Similarly, from the equality $[x_{i-1},x_{n+3-i}] +
[x_{i},x_{n+2-i}] = [[x_{i-1},x_{n+2-i}], x_1],$ for $3 \leq i
\leq k+1,$ by induction we  obtain
$$\begin{array}{l}
[x_{i},x_{n+2-i}] = (-1)^{i} \alpha_5 e_1 +
(-1)^{i}(n-5)\alpha_3 +(-1)^{i+1}\beta_{n-2}e_1 +\\[3mm]
\qquad+\sum\limits_{s=2}^{i-3}
(-1)^{s+i}\sum\limits_{t=1}^s Q_{s-t,n+1-2s}\sum\limits_{k=1}^{2s-2}\gamma_{s,n+1-2s+k}e_k
+ \sum\limits_{s=i-2}^{\lfloor\frac n
2\rfloor}(-1)^{s+i}\sum\limits_{t=1}^{i-3}Q_{s-t,n+1-2s}\sum\limits_{k=1}^{2s-2}\gamma_{s,n+1-2s+k}e_k+\\[3mm]
\qquad +\sum\limits_{s=0}^{\lfloor\frac {n+4-2i}2\rfloor}
(-1)^{s}Q_{s,n+5-2i-2s}\sum\limits_{k=1}^{2s+2i-6}\gamma_{i+s-2,n+5-2i-2s+k}e_k.
\end{array}$$

Taking $s'=s+i-2$ in the last sum, we obtain
$$\begin{array}{l}
\sum\limits_{s=0}^{\lfloor\frac {n+4-2i}2\rfloor}
(-1)^{s}Q_{s,n+5-2i-2s}\sum\limits_{k=1}^{2s+2i-6}\gamma_{i+s-2,n+5-2i-2s+k}e_k
= \\[3mm]
\qquad =\sum\limits_{s=i-2}^{\lfloor\frac {n}2\rfloor}
(-1)^{s+i}Q_{s-i+2,n+1-2s}\sum\limits_{k=1}^{2s-2}\gamma_{s,n+1-2s+k}e_k.
\end{array}$$

Placing this equality to above one we obtain the first equality of the lemma.

Now we will deduce the products $[x_{i},x_{n+3-i}].$

Using $[x_{3},x_n] = 0$ from the relation
$[x_{3},x_n] + [x_{4},x_{n-1}] = [[x_3,x_{n-1}], x_1]$ we get
$$[x_{4},x_{n-1}] =  - (n-5)\alpha_3 e_1
+  \sum\limits_{s=2}^{\lfloor\frac n 2\rfloor}
(-1)^{s+1}Q_{s-1,n+1-2s}\sum\limits_{k=1}^{2s-3}\gamma_{s,n+2-2s+k}e_k
$$

Using the products $[x_{i},x_{n+2-i}]$
 from the equality $[x_{i-1},x_{n+4-i}] +
[x_{i},x_{n+3-i}] = [[x_{i-1},x_{n+3-i}], x_1],$ for $4 \leq i
\leq k+2,$ by induction on $i$ similarly to the previous case we
obtain
$$\begin{array}{l}
[x_{i},x_{n+3-i}] = (-1)^{i+1} (i-3)(n-5)\alpha_3 e_1
+ \sum\limits_{s=2}^{i-3}
(-1)^{s+i+1}\sum\limits_{t=1}^s(i-2-t)Q_{s-t,n+1-2s}\sum\limits_{k=1}^{2s-3}\gamma_{s,n+2-2s+k}e_k+\\[3mm]
\qquad \quad+ \sum\limits_{s=i-2}^{\lfloor\frac n
2\rfloor}(-1)^{s+i+1}\sum\limits_{t=1}^{i-3}(i-2-t)Q_{s-t,n+1-2s}\sum\limits_{k=1}^{2s-3}\gamma_{s,n+2-2s+k}e_k.
\end{array}$$

The last equality from \eqref{eq4.4} follows by the induction on $p$ and $i$ (first by $p$, then by $i$.)
\end{proof}

Now we define some restrictions to the structure constants $\beta_i$ and $\gamma_{i,j}.$

From the equality \eqref{eq4.3} we obtain
$$[x_{\lfloor\frac n 2\rfloor+l}, x_{\lfloor\frac n 2\rfloor+l}]= \frac 1 2
[[x_{\lfloor\frac n 2\rfloor+l},x_{\lfloor\frac n 2\rfloor+l-1}],x_1], \quad 1 \leq l \leq
\lfloor\frac n 2\rfloor.
$$

\textbf{Let $n$ be even.} Then in the case of $l=1$  we get

$$[x_{\lfloor\frac n 2\rfloor+1}, x_{\lfloor\frac n 2\rfloor+1}]= \frac 1 2
[[x_{\lfloor\frac n 2\rfloor+1},x_{\lfloor\frac n 2\rfloor}],x_1]= \frac 1 2
\sum\limits_{k=1}^{n-2} \gamma_{\lfloor\frac n 2\rfloor,k+1}e_k
$$

On the other hand, from Lemma \ref{lemma3} we obtain
$$\begin{array}{l}
[x_{\lfloor\frac n 2\rfloor+1},x_{\lfloor\frac n 2\rfloor+1}] = (-1)^{\lfloor\frac n 2\rfloor+1} \alpha_5 e_1 +
(-1)^{\lfloor\frac n 2\rfloor+1}(n-5)\alpha_3 e_2+(-1)^{\lfloor\frac n 2\rfloor}\beta_{n-2}e_1+\\[2mm]
\quad +
\sum\limits_{s=2}^{\lfloor\frac n 2\rfloor-1} (-1)^{s+\lfloor\frac n 2\rfloor+1}\sum\limits_{t=1}^sQ_{s-t,n+1-2s}\sum\limits_{k=1}^{2s-2}\gamma_{s,n+1-2s+k}e_k
- \sum\limits_{t=1}^{\lfloor\frac n 2\rfloor-1}Q_{\lfloor\frac n 2\rfloor-t,1}\sum\limits_{k=1}^{n-2}\gamma_{\lfloor\frac n 2\rfloor,k+1}e_k.
\end{array}$$
%Therefore, we get
%$$ ( \alpha_5 - \beta_{n-2})e_1 +
%(n-5)\alpha_3 e_2+ \sum\limits_{s=2}^{\lfloor\frac n 2\rfloor-1}
%(-1)^{s}\sum\limits_{t=1}^sQ_{s-t,n+1-2s}\sum\limits_{k=1}^{2s-2}\gamma_{s,n+1-2s+k}e_k
%$$ $$+(-1)^{\lfloor\frac n 2\rfloor} \Big\{\frac
%1 2 + \sum\limits_{t=1}^{\lfloor\frac n 2\rfloor-1}Q_{[\frac n
%2]-t,1}\Big\}\sum\limits_{k=1}^{n-2}\gamma_{\lfloor\frac n 2\rfloor,k+1}e_k=0.
%$$
%
%
%$$ ( \alpha_5 - \beta_{n-2})e_1 +
%(n-5)\alpha_3 e_2+ $$$$\sum\limits_{k=1}^{n-2}
%\Big\{\sum\limits_{s=[\frac {k+3} 2]}^{\lfloor\frac n 2\rfloor}
%(-1)^{s}\gamma_{s,n+1-2s+k}
%\Big(\sum\limits_{t=1}^{s}Q_{s-t,n+1-2s}\Big)-(-1)^{[\frac n
%2]}\frac 1 2 \gamma_{\lfloor\frac n 2\rfloor,k+1} \Big\}e_k =0$$
Comparing the coefficients at the basis elements we obtain
the following restrictions:

\begin{equation}\label{eq8}\left\{\begin{array}{cc}\alpha_5 - \beta_{n-2}+
\sum\limits_{s=2}^{\lfloor\frac n 2\rfloor} (-1)^{s}\gamma_{s,n+2-2s}
\Big(\sum\limits_{t=1}^{s}Q_{s-t,n+1-2s}\Big)-(-1)^{\lfloor\frac n 2\rfloor}\frac 1 2 \gamma_{\lfloor\frac n 2\rfloor,2} =0,&\\[2mm]
(n-5)\alpha_3 + \sum\limits_{s=2}^{\lfloor\frac n 2\rfloor}
(-1)^{s}\gamma_{s,n+3-2s}
\Big(\sum\limits_{t=1}^{s}Q_{s-t,n+1-2s}\Big)-(-1)^{\lfloor\frac n 2\rfloor}\frac 1 2 \gamma_{\lfloor\frac n 2\rfloor,3} =0,&\\[2mm]
\sum\limits_{s=\lfloor\frac {k+3} 2\rfloor}^{\lfloor\frac n 2\rfloor}
(-1)^{s}\gamma_{s,n+1-2s+k}
\Big(\sum\limits_{t=1}^{s}Q_{s-t,n+1-2s}\Big)-(-1)^{\lfloor\frac n 2\rfloor}\frac 1 2 \gamma_{\lfloor\frac n 2\rfloor,k+1} =0, &3 \leq k \leq
n-2.\end{array}\right.
\end{equation}

If $l \geq 2,$ then we have
$$\begin{array}{l}
 [x_{\lfloor\frac n 2\rfloor+l},x_{\lfloor\frac n 2\rfloor+l}]
=\frac 1 2 [[x_{\lfloor\frac n 2\rfloor+l},x_{\lfloor\frac n 2\rfloor+l-1}],x_1]= \\[3mm]
\qquad = \frac 1 2\sum\limits_{s=l+1}^{\lfloor\frac n 2\rfloor-l+1}
(-1)^{s+\lfloor\frac n 2\rfloor+l-1}\sum\limits_{t=1}^s\Big(\begin{array}{c} \lfloor\frac n 2\rfloor+l-2-t\\
2l-3\end{array}\Big)
Q_{s-t,n+1-2s}\sum\limits_{k=1}^{2s-2l}\gamma_{s,n+2l-1-2s+k}e_k +\\[3mm]
\qquad +\frac 1 2 \sum\limits_{s=\max\{\lfloor\frac n 2\rfloor-l+2; l+1\}}^{\lfloor\frac n 2\rfloor}(-1)^{s+\lfloor\frac n 2\rfloor+l-1}
\sum\limits_{t=1}^{\lfloor\frac n 2\rfloor-l+1}\Big(\begin{array}{c} \lfloor\frac n 2\rfloor+l-2-t\\
2l-3\end{array}\Big)Q_{s-t,n+1-2s}\sum\limits_{k=1}^{2s-2l}\gamma_{s,n+2l-1-2s+k}e_k.
\end{array}$$

On the other hand, in the equality \eqref{eq4.4} for $p=2l$ and
$i=\lfloor\frac n 2\rfloor+l$ we deduce
$$\begin{array}{l}
 [x_{\lfloor\frac n 2\rfloor+l},x_{\lfloor\frac n 2\rfloor+l}] =
\sum\limits_{s=l+1}^{\lfloor\frac n 2\rfloor-l}
(-1)^{s+\lfloor\frac n 2\rfloor+l}\sum\limits_{t=1}^s\Big(\begin{array}{c} \lfloor\frac n 2\rfloor+l-2-t\\
2l-2\end{array}\Big)
Q_{s-t,n+1-2s}\sum\limits_{k=1}^{2s-2l}\gamma_{s,n+2l-1-2s+k}e_k +\\[3mm]
\qquad +\sum\limits_{s=\max\{\lfloor\frac n 2\rfloor-l+1; l+1\}}^{\lfloor\frac n 2\rfloor}(-1)^{s+\lfloor\frac n 2\rfloor+l}\sum\limits_{t=1}^{\lfloor\frac n 2\rfloor-l}
\Big(\begin{array}{c} \lfloor\frac n 2\rfloor+l-2-t\\
2l-2\end{array}\Big)Q_{s-t,n+1-2s}\sum\limits_{k=1}^{2s-2l}\gamma_{s,n+2l-1-2s+k}e_k.
\end{array}$$

%Comparing this equalities we obtain
%$$ \sum\limits_{s=l+1}^{\lfloor\frac n 2\rfloor-l}
%(-1)^{s}\sum\limits_{t=1}^s\lfloor\Big(\begin{array}{c} \lfloor\frac n 2\rfloor+l-2-t\\
%2l-2\end{array}\Big)  + \frac 1 2 \Big(\begin{array}{c} \lfloor\frac n 2\rfloor+l-2-t\\
%2l-3\end{array}\Big)\rfloorQ_{s-t,n+1-2s}\sum\limits_{k=1}^{2s-2l}\gamma_{s,n+2l-1-2s+k}e_k
%+
%$$
%$$
%\sum\limits_{s=\max\{\lfloor\frac n 2\rfloor-l+1; l+1\}}^{\lfloor\frac n 2\rfloor} \lfloor
%\frac 1 2
%Q_{s-\lfloor\frac n 2\rfloor+l-1, n+1-2s} + \sum\limits_{t=1}^{\lfloor\frac n 2\rfloor-l}\Big\{\Big(\begin{array}{c} \lfloor\frac n 2\rfloor+l-2-t\\
%2l-2\end{array}\Big) + \frac 1 2 \Big(\begin{array}{c} \lfloor\frac n 2\rfloor+l-2-t\\
%2l-3\end{array}\Big) \Big\}  Q_{s-t,n+1-2s} \rfloor \times $$$$
%(-1)^{s}\sum\limits_{k=1}^{2s-2l}\gamma_{s,n+2l-1-2s+k}e_k= 0.$$

%$$ \sum\limits_{k=1}^{n-2l}
%\Big\{\sum\limits_{s=\lfloor\frac {k+1} 2\rfloor+l}^{\lfloor\frac n 2\rfloor}
%(-1)^{s}\gamma_{s,n+2l-1-2s+k}\sum\limits_{t=1}^{\min\{s,\lfloor\frac n 2\rfloor-l\}}
% \lfloor\Big(\begin{array}{c} \lfloor\frac n 2\rfloor+l-2-t\\
%2l-2\end{array}\Big)  + \frac 1 2 \Big(\begin{array}{c} \lfloor\frac n 2\rfloor+l-2-t\\
%2l-3\end{array}\Big)\rfloorQ_{s-t,n+1-2s} +$$
%$$\sum\limits_{s=\max\{\lfloor\frac n 2\rfloor-l+1; \lfloor\frac {k+1} 2\rfloor+l\}}^{\lfloor\frac n 2\rfloor}
%\frac 1 2(-1)^{s}\gamma_{s,n+2l-1-2s+k}Q_{s-\lfloor\frac n 2\rfloor+l-1,
%n+1-2s}\Big\}e_k =0$$

Comparing the coefficients at the basis elements we get
\begin{equation}\label{eq9}\small\left\{\begin{array}{l}
\sum\limits_{s=\lfloor\frac {k+1} 2\rfloor+l}^{\lfloor\frac n 2\rfloor}
(-1)^{s}\gamma_{s,n+2l-1-2s+k}\sum\limits_{t=1}^{\min\{s,\lfloor\frac n 2\rfloor-l\}}
 \left [\Big(\begin{array}{c} \lfloor\frac n 2\rfloor+l-2-t\\
2l-2\end{array}\Big)  + \frac 1 2 \Big(\begin{array}{c} \lfloor\frac n 2\rfloor+l-2-t\\
2l-3\end{array}\Big)\right ] Q_{s-t,n+1-2s} +\\[3mm]
\sum\limits_{s=\max\{\lfloor\frac n 2\rfloor-l+1; \lfloor\frac {k+1} 2\rfloor+l\}}^{\lfloor\frac n 2\rfloor}
\frac 1 2(-1)^{s}\gamma_{s,n+2l-1-2s+k}Q_{s-\lfloor\frac n 2\rfloor+l-1,
n+1-2s}=0,\end{array}\right.\end{equation}
where $ 2 \leq l \leq
\lfloor\frac n 2\rfloor-1, \ 1 \leq k \leq n-2l.$

\

\textbf{Let $n$ be odd.} Then in the case of $l=2$ we get $
[x_{\lfloor\frac n 2\rfloor+2},x_{\lfloor\frac n 2\rfloor+2}]=\frac 1 2 [[x_{\lfloor\frac n 2\rfloor+2},x_{\lfloor\frac n 2\rfloor+1}],x_1]$  and using the first equality of
\eqref{eq4.4} we obtain $$[x_{\lfloor\frac n 2\rfloor+2},x_{\lfloor\frac n 2\rfloor+2}]=
\frac 1 2 (-1)^{\lfloor\frac n 2\rfloor} (n-5) \alpha_3e_1+
 \frac 1 2\sum\limits_{s=2}^{\lfloor\frac n 2\rfloor}(-1)^{s+\lfloor\frac n 2\rfloor}\sum\limits_{t=1}^sQ_{s-t,n+1-2s}
\sum\limits_{k=1}^{2s-3}\gamma_{s,n+2-2s+k}e_k.$$

On the other hand, from the second equality of \eqref{eq4.4} for $i=\lfloor\frac n 2\rfloor+2$ we get

$$\begin{array}{lll}
[x_{\lfloor\frac n 2\rfloor+2},x_{\lfloor\frac n 2\rfloor+2}]&=& (-1)^{\lfloor\frac n 2\rfloor+1} (\lfloor\frac n 2\rfloor-1)(n-5)\alpha_3 e_1
+\\[3mm]
&+& \sum\limits_{s=2}^{\lfloor\frac n 2\rfloor-1} (-1)^{s+\lfloor\frac n 2\rfloor+1}\sum\limits_{t=1}^s(\lfloor\frac n 2\rfloor-t)Q_{s-t,n+1-2s}
\sum\limits_{k=1}^{2s-3}\gamma_{s,n+2-2s+k}e_k-\\[3mm]
&-&\sum\limits_{t=1}^{\lfloor\frac n 2\rfloor-1}(\lfloor\frac n 2\rfloor-t)Q_{\lfloor\frac n 2\rfloor-t,2}\sum\limits_{k=1}^{n-4}\gamma_{\lfloor\frac n 2\rfloor,k+3}e_k.
\end{array}$$

 %$$ (\lfloor\frac n 2\rfloor-\frac
%1 2)(n-5)\alpha_3 e_1 + \sum\limits_{s=2}^{\lfloor\frac n 2\rfloor-1}
%(-1)^{s}\sum\limits_{t=1}^s(\lfloor\frac n 2\rfloor-t+\frac 1
%2)Q_{s-t,n+1-2s} \sum\limits_{k=1}^{2s-3}\gamma_{s,n+2-2s+k}e_k$$
%$$ +(-1)^{\lfloor\frac n 2\rfloor}\Big(\sum\limits_{t=1}^{\lfloor\frac n 2\rfloor-1}(\lfloor\frac n
%2\rfloor-t+\frac 1 2)Q_{\lfloor\frac n 2\rfloor-t,2}+  \frac 1 2
%Q_{0,2}\Big)\sum\limits_{k=1}^{n-4}\gamma_{[\frac n
%2],k+3}e_k=0.$$
%$$ (\lfloor\frac n 2\rfloor-\frac
%1 2)(n-5)\alpha_3 e_1 +
%\sum\limits_{k=1}^{n-4}\Big(\sum\limits_{s=[\frac {k+4}
%2]}^{[\frac {n} 2]}
%(-1)^{s}\gamma_{s,n+2-2s+k}\sum\limits_{t=1}^s(\lfloor\frac n
%2\rfloor-t+\frac 1 2)Q_{s-t,n+1-2s} \Big)e_k$$
%
%
%Therefore, we get
Comparing the coefficients at the basis elements we derive
\begin{equation}\label{eq10}\left\{\begin{array}{ll} (\lfloor\frac n 2\rfloor-\frac 1 2)(n-5)\alpha_3 +
\sum\limits_{s=2}^{\lfloor\frac {n} 2\rfloor}
(-1)^{s}\gamma_{s,n+3-2s}\sum\limits_{t=1}^s(\lfloor\frac n
2\rfloor-t+\frac 1 2)Q_{s-t,n+1-2s}=0, & \\[1mm]
\sum\limits_{s=\lfloor\frac {k+4} 2\rfloor}^{\lfloor\frac {n} 2\rfloor}
(-1)^{s}\gamma_{s,n+2-2s+k}\sum\limits_{t=1}^s(\lfloor\frac n
2\rfloor-t+\frac 1 2)Q_{s-t,n+1-2s} =0, & 2 \leq k \leq
n-4.\end{array}\right.\end{equation}

If $l \geq 3,$ then we have
$$\begin{array}{l}
 [x_{\lfloor\frac n 2\rfloor+l},x_{\lfloor\frac n 2\rfloor+l}] =\frac 1 2 [[x_{\lfloor\frac n 2\rfloor+l},x_{\lfloor\frac n 2\rfloor+l-1}],x_1]= \\[3mm]
\qquad = \frac 1 2\sum\limits_{s=l}^{\lfloor\frac n 2\rfloor-l+2}
(-1)^{s+\lfloor\frac n 2\rfloor+l}\sum\limits_{t=1}^s\Big(\begin{array}{c} \lfloor\frac n 2\rfloor+l-2-t\\
2l-4\end{array}\Big)
Q_{s-t,n+1-2s}\sum\limits_{k=1}^{2s-2l+1}\gamma_{s,n+2l-2-2s+k}e_k
+\\[3mm]
\qquad +\frac 1 2 \sum\limits_{s=\max\{\lfloor\frac n 2\rfloor-l+3; l\}}^{\lfloor\frac n 2\rfloor}(-1)^{s+\lfloor\frac n 2\rfloor+l}
\sum\limits_{t=1}^{\lfloor\frac n 2\rfloor-l+2}\Big(\begin{array}{c} \lfloor\frac n 2\rfloor+l-2-t\\
2l-4\end{array}\Big)Q_{s-t,n+1-2s}\sum\limits_{k=1}^{2s-2l+1}\gamma_{s,n+2l-2-2s+k}e_k.
\end{array}$$

On the other hand, in the equality \eqref{eq4.4} for $p=2l-1$ and
$i=\lfloor\frac n 2\rfloor+l$ we have
$$ \begin{array}{l}
[x_{\lfloor\frac n 2\rfloor+l},x_{\lfloor\frac n 2\rfloor+l}] =
\sum\limits_{s=l}^{\lfloor\frac n 2\rfloor-l+1}
(-1)^{s+\lfloor\frac n 2\rfloor+l-1}\sum\limits_{t=1}^s\Big(\begin{array}{c} \lfloor\frac n 2\rfloor+l-2-t\\
2l-3\end{array}\Big)
Q_{s-t,n+1-2s}\sum\limits_{k=1}^{2s-2l+1}\gamma_{s,n+2l-2-2s+k}e_k
+\\[3mm]
\qquad +
\sum\limits_{s=\max\{\lfloor\frac n 2\rfloor-l+2; l\}}^{\lfloor\frac n 2\rfloor}(-1)^{s+\lfloor\frac n 2\rfloor+l-1}\sum\limits_{t=1}^{\lfloor\frac n 2\rfloor-l+1}
\Big(\begin{array}{c} \lfloor\frac n 2\rfloor+l-2-t\\
2l-3\end{array}\Big)Q_{s-t,n+1-2s}\sum\limits_{k=1}^{2s-2l+1}\gamma_{s,n+2l-2-2s+k}e_k.
\end{array}$$

%Comparing this equalities we obtain
%$$ \sum\limits_{s=l}^{\lfloor\frac n 2\rfloor-l+1}
%(-1)^{s}\sum\limits_{t=1}^s\lfloor\Big(\begin{array}{c} \lfloor\frac n 2\rfloor+l-2-t\\
%2l-3\end{array}\Big)  + \frac 1 2 \Big(\begin{array}{c} \lfloor\frac n 2\rfloor+l-2-t\\
%2l-4\end{array}\Big)\rfloorQ_{s-t,n+1-2s}\sum\limits_{k=1}^{2s-2l+1}\gamma_{s,n+2l-2-2s+k}e_k
%+
%$$
%$$
%\sum\limits_{s=\max\{\lfloor\frac n 2\rfloor-l+2; l\}}^{\lfloor\frac n 2\rfloor} \lfloor
%\frac 1 2
%Q_{s-\lfloor\frac n 2\rfloor+l-2, n+1-2s} + \sum\limits_{t=1}^{\lfloor\frac n 2\rfloor-l+1}\Big\{\Big(\begin{array}{c} \lfloor\frac n 2\rfloor+l-2-t\\
%2l-3\end{array}\Big) + \frac 1 2 \Big(\begin{array}{c} \lfloor\frac n 2\rfloor+l-2-t\\
%2l-4\end{array}\Big) \Big\}  Q_{s-t,n+1-2s} \rfloor \times $$$$
%(-1)^{s}\sum\limits_{k=1}^{2s-2l+1}\gamma_{s,n+2l-2-2s+k}e_k= 0.$$

Comparing the coefficients at the basis elements we get
%$$ \sum\limits_{k=1}^{n-2l}
%\Big\{\sum\limits_{s=[\frac {k} 2]+l}^{\lfloor\frac n 2\rfloor}
%(-1)^{s}\gamma_{s,n+2l-2-2s+k}\sum\limits_{t=1}^{\min\{s,[\frac n
%2]-l+1\}}
% \lfloor\Big(\begin{array}{c} \lfloor\frac n 2\rfloor+l-2-t\\
%2l-3\end{array}\Big)  + \frac 1 2 \Big(\begin{array}{c} \lfloor\frac n 2\rfloor+l-2-t\\
%2l-4\end{array}\Big)\rfloorQ_{s-t,n+1-2s} +$$
%$$\sum\limits_{s=\max\{\lfloor\frac n 2\rfloor-l+2; [\frac {k} 2]+l\}}^{\lfloor\frac n 2\rfloor}
%\frac 1 2(-1)^{s}\gamma_{s,n+2l-2-2s+k}Q_{s-\lfloor\frac n 2\rfloor+l-2,
%n+1-2s}\Big\}e_k =0$$

\begin{equation}\label{eq11}\small\left\{\begin{array}{l} \sum\limits_{s=\lfloor\frac {k} 2\rfloor+l}^{\lfloor\frac n 2\rfloor}
(-1)^{s}\gamma_{s,n+2l-2-2s+k}\sum\limits_{t=1}^{\min\{s,\lfloor\frac n 2\rfloor-l+1\}}
 \left [\Big(\begin{array}{c} \lfloor\frac n 2\rfloor+l-2-t\\
2l-3\end{array}\Big)  + \frac 1 2 \Big(\begin{array}{c} \lfloor\frac n 2\rfloor+l-2-t\\
2l-4\end{array}\Big)\right ] Q_{s-t,n+1-2s} +\\[1mm]
$$\sum\limits_{s=\max\{\lfloor\frac n 2\rfloor-l+2; \lfloor\frac {k} 2\rfloor+l\}}^{\lfloor\frac n 2\rfloor}
\frac 1 2(-1)^{s}\gamma_{s,n+2l-2-2s+k}Q_{s-\lfloor\frac n 2\rfloor+l-2,
n+1-2s}=0, \end{array}\right.
\end{equation}
where $3 \leq l \leq \lfloor\frac n 2\rfloor, \ 1  \leq k \leq n-2l.$

Therefore, we obtain following main Theorem of this section.

\begin{teo}
Let $L$ be a Leibniz algebra such that $L/ I \cong n_{n,1}$ and
$I$ is the $L/ I$-module with the minimal faithful representation.
Then $L$ admits a basis $\{ x_1, x_2, \dots, x_n, e_1, e_2, \dots,
e_{n}\}$ such that the multiplications table for this
basis has the form \eqref{eq1}, \eqref{eq4.1}, \eqref{eq4.2},
\eqref{eq4.4} with the restrictions \eqref{eq8}, \eqref{eq9},
\eqref{eq10}, \eqref{eq11}.
\end{teo}

Now we are in position to give a classification  of such algebras up to isomorphism for the case $n=4$.

In this case we get the following family of algebras denoted by $\mu(\alpha_1,\alpha_2,\alpha_3,\alpha_4,\beta_1,\beta_2,\gamma_1,\gamma_2):$
$$\left\{\begin{array}{ll}
[e_{2}, x_1 ] = e_1,& [x_3,x_1]=x_4 - \alpha_{2} e_3, \\[2mm]
[e_{3},x_1 ] = e_{2},& [x_4,x_1]= -\alpha_{4} e_1 -2\alpha_{2}
e_2,\\[2mm] [e_{4}, x_2] =
e_3,& [x_2,x_2] = \beta_{1}  e_1+\beta_{2}  e_2,\\[2mm]
{} [e_{4}, x_3]=
e_2,& [x_3,x_2] = \gamma_{1}  e_1+\gamma_{2}  e_2 -2\alpha_{3}e_3,\\[2mm]
{} [e_{4}, x_4] =
e_1, & [x_4,x_3] = - \alpha_{3}  e_1,\\[2mm]
{} [x_1,x_2]=-x_3, &  [x_3,x_3] = \frac 1 2\gamma_{2}  e_1-  \alpha_{3} e_2,\\[2mm]
{} [x_1,x_3]=-x_4, & [x_4,x_2] = \frac 1 2\gamma_{2}  e_1-  \alpha_{3} e_2\\[2mm]
{} [x_1,x_1]=\alpha_{1} e_3+\alpha_{2} e_4, & [x_2,x_3] =
(\beta_{2} - \gamma_{1})e_1- \gamma_{2}e_2 + \alpha_{3}e_3,\\[2mm]
{} [x_2,x_1]=x_3+\alpha_{3} e_4, & [x_2,x_4] =
 - \frac 3 2 \gamma_{2}e_1-  \alpha_{3}e_2,\\[2mm]
{} [x_1,x_4]=\alpha_{4} e_1+\alpha_{2} e_2.
\end{array}\right.$$

\begin{teo}
Let $L$ be an $8$-dimensional Leibniz algebra such that $L/ I \cong n_{4,1}$ and
$I$ is the $L/ I$-module with the minimal faithful representation.
Then $L$ is isomorphic to the one of the following pairwise non isomorphic algebras:
$$\begin{array}{|c|c|c|c|}
\hline
\mu(\alpha_1,\alpha_2,1,1,\beta_1,\beta_2,0,1)&\mu(\alpha_1,1,1,0,\beta_1,\beta_2,0,1)&\mu(1,0,1,0,\beta_1,\beta_2,0,1)&\mu(\alpha_1,1,1,1,\beta_1,\beta_2,0,0)\\
\hline
\mu(1,0,1,1,\beta_1,\beta_2,0,0)&\mu(0,0,1,1,1,\beta_2,0,0)&\mu(0,0,1,1,0,1,0,0)&\mu(0,0,1,1,0,0,0,0)\\
\hline
\mu(1,1,1,0,\beta_1,\beta_2,0,0)&\mu(0,1,1,0,1,\beta_2,0,0)&\mu(0,1,1,0,0,1,0,0)&\mu(0,1,1,0,0,0,0,0)\\
\hline
\mu(1,0,1,0,1,\beta_2,0,0)&\mu(1,0,1,0,0,1,0,0)&\mu(0,0,1,0,1,\beta_2,0,0)&\mu(0,0,1,0,0,1,0,0)\\
\hline
\mu(0,0,1,0,0,0,0,0)&\mu(\alpha_1,1,0,1,0,\beta_2,\gamma_1,1)&\mu(\alpha_1,0,0,1,0,1,\gamma_1,1)&\mu(\alpha_1,0,0,1,0,0,1,1)\\
\hline
\mu(\alpha_1,0,0,1,0,0,0,1)&\mu(\alpha_1,1,0,0,0,1,\gamma_1,1)&\mu(\alpha_1,1,0,0,0,0,1,1)&\mu(\alpha_1,1,0,0,0,0,0,1)\\
\hline
\mu(1,0,0,0,0,1,\gamma_1,1)&\mu(1,0,0,0,0,0,1,1)&\mu(1,0,0,0,0,0,0,1)&\mu(0,0,0,0,0,1,\gamma_1,1)\\
\hline
\mu(0,0,0,0,0,0,1,1)&\mu(0,0,0,0,0,0,0,1)&\mu(1,1,0,1,\beta_1,\beta_2,\gamma_1,0)&\mu(0,1,0,1,1,\beta_2,\gamma_1,0)\\
\hline
\mu(0,1,0,1,0,1,\gamma_1,0)&\mu(0,1,0,1,0,0,1,0)&\mu(0,1,0,1,0,0,0,0)&\mu(1,0,0,1,1,\beta_2,\gamma_1,0)\\
\hline
\mu(1,0,0,1,0,1,\gamma_1,0)&\mu(1,0,0,1,0,0,1,0)&\mu(1,0,0,1,0,0,0,0)&\mu(0,0,0,1,1,1,\gamma_1,0)\\
\hline
\mu(0,0,0,1,1,0,1,0)&\mu(0,0,0,1,1,0,0,0)&\mu(0,0,0,1,0,1,1,0)&\mu(0,0,0,1,0,1,0,0)\\
\hline
\mu(0,0,0,1,0,0,1,0)&\mu(0,0,0,1,0,0,0,0)&\mu(1,1,0,0,1,\beta_2,\gamma_1,0)&\mu(1,1,0,0,0,1,\gamma_1,0)\\
\hline
\mu(1,1,0,0,0,0,1,0)&\mu(1,1,0,0,0,0,0,0)&\mu(1,0,0,0,1,1,\gamma_1,0)&\mu(1,0,0,0,1,0,1,0)\\
\hline
\mu(1,0,0,0,1,0,0,0)&\mu(1,0,0,0,0,1,1,0)&\mu(1,0,0,0,0,1,0,0)&\mu(1,0,0,0,0,0,1,0)\\
\hline
\mu(1,0,0,0,0,0,0,0)&\mu(0,1,0,0,1,1,\gamma_1,0)&\mu(0,1,0,0,1,0,1,0)&\mu(0,1,0,0,1,0,0,0)\\
\hline
\mu(0,1,0,0,0,1,\gamma_1,0)&\mu(0,1,0,0,0,0,1,0)&\mu(0,1,0,0,0,0,0,0)&\mu(0,0,0,0,1,1,\gamma_1,0)\\
\hline
\mu(0,0,0,0,0,1,\gamma_1,0)&\mu(0,0,0,0,1,0,1,0)&\mu(0,0,0,0,1,0,0,0)&\mu(0,0,0,0,0,0,1,0)\\
\hline
\mu(0,0,0,0,0,0,0,0)& & &\\
\hline
\end{array}$$
with $\alpha_1,\alpha_2,\beta_1,\beta_2,\gamma_1\in \mathbb{C}.$
\end{teo}

\begin{proof}
Let $L$ be an 8-dimensional Leibniz algebra given by $\mu(\alpha_1,\alpha_2,\alpha_3,\alpha_4,\beta_1,\beta_2,\gamma_1,\gamma_2).$
We make the following change of basis:
$$\begin{array}{l}
x'_1=\displaystyle\sum_{k=1}^{4} P_k x_k+\sum_{k=1}^{4} Q_k e_k,\\  [2mm]
x'_2=\displaystyle\sum_{k=1}^{4} M_k x_k+\sum_{k=1}^{4} N_k e_k,\\   [2mm]
e'_4=\displaystyle\sum_{k=1}^{4} R_k x_k+\sum_{k=1}^{4} T_k e_k,
\end{array}$$
while the other elements of the new basis (i.e. $e_1', e_2', e_3', x_3'$ and $x_4'$) are obtained as products of the above elements.

The table of multiplication in this new basis implies the following restrictions on the coefficients:

$$\begin{array}{l}
P_2=M_1=R_k=0,\qquad 1\leq k\leq 4,\\[3mm]
T_3 = -\displaystyle\frac{T_4 P_3}{P_1},\qquad \quad T_2=-\displaystyle\frac{T_4 P_4}{P_1},\qquad \quad N_4=\alpha_3 M_3,\\[3mm]
Q_4=\displaystyle\frac{\alpha_2 P_1 M_3}{M_2},\\[3mm]
Q_3=-\displaystyle\frac{-\alpha_3 P_3^2 M_2-\alpha_1 P_1^2 M_3+\alpha_2 P_1 P_3 M_3}{P_1 M_2},\\[3mm]
Q_2=-\displaystyle\frac{-2\alpha_2 T_1P_1^2 M_2+\gamma_2 T_4 P_3^2 M_2-2\alpha_3 T_4 P_3P_4M_2+
2\alpha_2 T_4 P_1P_4 M_3-2\alpha_1 T_4 P_1^2M_4}{2 T_4P_1M_2},\\[3mm]
N_3=-\displaystyle\frac{\alpha_3 P_4M_2^2-\alpha_3 P_3 M_2 M_3+\alpha_2 P_1 M_3^2-\alpha_2P_1 M_2 M_4}{P_1M_2},\\[3mm]
N_2=-\displaystyle\frac{-\alpha_3 T_1P_1M_2^2+\beta_2 T_4P_3M_2^2-\gamma_2 T_4 P_4 M_2^2+\gamma_2T_4 P_3M_2M_3-\alpha_3 T_4 P_3 M_2M_4+\alpha_2 T_4P_1 M_3M_4}{T_4P_1M_2},\\[3mm]
T_4P_1M_2\neq 0.
\end{array}$$

Calculating new parameters we obtain:
$$\begin{array}{l}
\alpha'_1=\displaystyle\frac{\alpha_1 P_1^2}{T_4 M_2},\qquad \quad \alpha'_2=\displaystyle\frac{\alpha_2 P_1^2}{T_4},\\[3mm]
\alpha'_3=\displaystyle\frac{\alpha_3 P_1 M_2}{T_4},\qquad \quad \alpha'_4=\displaystyle\frac{\alpha_4 P_1}{T_4},\\[3mm]
\beta'_1=\displaystyle\frac{2\beta_1M_2^2+\gamma_2 M_3^2-2\gamma_2M_2M_4}{2T_4P_1^2 M_2},\\[3mm]
\gamma'_1=\displaystyle\frac{\gamma_1M_2^2-\alpha_3M_3^2+2\alpha_3M_2M_4}{T_4 P_1M_2},\\[3mm]
\beta'_2=\displaystyle\frac{\beta_2 M_2}{T_4 P_1},\qquad \quad \gamma'_2=\displaystyle\frac{\gamma_2 M_2}{T_4}.
\end{array}$$

Considering all the possible cases we obtain the families of algebras listed in the theorem.
\end{proof}

\end{document}